\documentclass[10pt]{article}

\usepackage{amssymb, latexsym, epsfig, amsmath, amsthm, setspace}

\usepackage[T1]{fontenc}

\usepackage{todonotes}
\usepackage{enumerate}

\newtheorem{theorem}{Theorem}
\newtheorem{lemma}{Lemma}
\newtheorem{corollary}{Corollary}

\begin{document}

\title{Critical Pebbling Numbers of Graphs}

\author{Courtney R. Gibbons \\
Mathematics Department \\ Hamilton College \\ 198 College Hill Road \\ Clinton, NY 13323 \\ crgibbon@hamilton.edu \and Joshua D. Laison \\ Mathematics Department \\ Willamette University\\ 900 State St., Salem, OR 97301 \\ jlaison@willamette.edu \and Erick J. Paul \\ Beckman Institute for Advanced Science and Technology \\ University of Illinois at Urbana-Champaign \\ 405 N Mathews Ave. \\ Urbana, IL 61801 \\ ejpaul@illinois.edu}

\maketitle

\begin{abstract}
We define three new pebbling parameters of a connected graph $G$,
the \textit{$r$-}, \textit{$g$-}, and \textit{$u$-critical
pebbling numbers}.  Together with the pebbling number, the optimal pebbling number, the number
of vertices $n$ and the diameter $d$ of the graph, this yields 7 graph parameters.  We determine the relationships between these parameters.  We investigate properties of the $r$-critical pebbling number, and distinguish between greedy graphs, thrifty graphs, and graphs for which the $r$-critical pebbling number is $2^d$.
\end{abstract}

\noindent \textbf{Keywords}: pebbling, optimal pebbling number, critical pebbling number

\noindent \textbf{Mathematics Subject Classification}: 05C57, 05C75

\section{Pebbling Numbers}

Let $G$ be a connected graph.  A \textit{pebbling distribution}
(or simply \textit{distribution}) $D$ on $G$ is a function which
assigns to each vertex of $G$ a non-negative integer number of pebbles.  If $D$
is a distribution on a graph $G$ and $a$ is a vertex of $G$, we
denote by $D(a)$ the number of pebbles on $a$ in the distribution
$D$.  The \textit{size} of the distribution $D$ is the number of
pebbles in $D$, $|D|=\sum_{a \in V(G)} D(a)$.

A \textit{pebbling step} $[a,b]$ is an operation which takes the
distribution $D$, removes two pebbles from the vertex $a$, and
adds one pebble at the adjacent vertex $b$.  A distribution $D$ is
\textit{$r$-solvable} if there exists a sequence of pebbling steps
starting with $D$ and ending with at least one pebble on the
vertex $r$, and \textit{solvable} if $D$ is $r$-solvable for all
$r$.  We call such a sequence a \textit{solution} of $D$.  A distribution $D$ is \textit{$r$-unsolvable} if it is not
$r$-solvable, and \textit{unsolvable} if there is some vertex $r$ for which $D$ is not $r$-solvable.

A \textit{rooted distribution} is a distribution which also
fixes the vertex $r$ (the \textit{root vertex}).
Note that for a rooted distribution, the terms solvable and
$r$-solvable are interchangeable.  In general, any statement about a distribution $D$
can be applied to a corresponding rooted distribution as well.  For emphasis, we say
that an un-rooted distribution is a \textit{global} distribution.

The \textit{pebbling number} $p(G)$ is the minimum number such
that any distribution on $G$ with $p(G)$ pebbles is solvable
\cite{Hurlbert99}. For example, since the first distribution on
the graph $C_7$ in Figure \ref{C7} is unsolvable, $p(C_7)>10$.  In
fact, $p(C_7)=11$ \cite{Hurlbert99}.

\begin{figure}[!ht]
\begin{center}
\includegraphics[width=\textwidth]{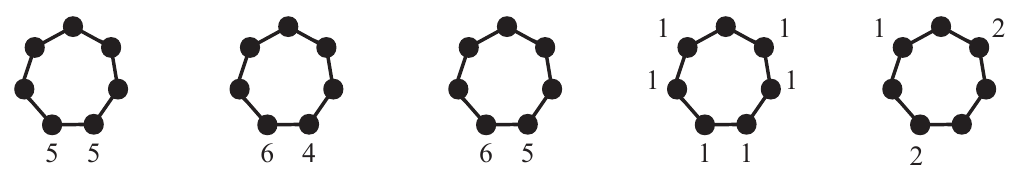}

\caption{Five pebbling distributions on the graph $C_7$.  The second and third distributions are solvable.  The rest are unsolvable.}
\label{C7}
\end{center}
\end{figure}

Pebbling on graphs can be used to model the distribution of resources with loss.  For example, we might be shipping water across roadways, and the water evaporates as it travels; or we might be sending heated water through pipes, and the water cools as it travels; or we might be sending cell phone signals from towers, and the signal weakens as it travels.  The pebbling number of the distribution network is then a measure of how much of the resource we need to have one unit arrive at our destination.

The pebbling number can also be thought of in the context of a \textit{pebbling game} on the graph $G$ that we play against an opponent.  First the  opponent distributes some pebbles on the vertices of $G$, and picks a root vertex.  Then we make pebbling moves to put at least one pebble on the root.  The pebbling number $p(G)$ is the smallest number of pebbles with which we can always win the pebbling game.

The pebbling numbers of several families of graphs have been found, such as trees \cite{Bunde08,Chung89} and grid graphs \cite{Clarke97}, and significant progress has been made on diameter $k$ graphs \cite{Bukh06,Postle14,Postle13}.  Finding the pebbling number of an arbitrary connected graph is computationally difficult, but can be done in polynomial time for several additional families of graphs \cite{Bekmetjev09,Cusack12,Herscovici13,Lewis14}.  Much research has focused on the pebbling number of a product of graphs, and in particular on Graham's Conjecture, which states that $p(G \times H) \leq p(G) \times p(H)$ \cite{Feng01,Herscovici03,Herscovici98}.

In addition, many variations of pebbling numbers have been defined, including the optimal pebbling number \cite{Bunde08,Herscovici11}, the $t$-pebbling number \cite{Gao13,Herscovici13}, the cover pebbling number \cite{Godbole09}, the domination cover pebbling number \cite{Gardner08}, and the weighted pebbling number \cite{Jones14}.  These vary the rules for what constitutes a pebbling step, what constitutes a solved distribution of pebbles, or the turns in the pebbling game. For example, the cover pebbling number requires at least one pebble on each vertex of $G$ in a solved distribution, and in the optimal pebbling game we distribute the pebbles on $G$, then the  opponent picks the root.  Recently Hurlbert established a general pebbling framework that includes several of these variations \cite{Hurlbert13}.

In this paper we introduce three new pebbling parameters, the $r$-, $g$-, and $u$-critical pebbling numbers.  Like the optimal pebbling number, these variations can all be thought of as variations in the rules of the pebbling game, as we outline below.  We first give more formal definitions.

A rooted distribution $D$ is \textit{minimally $r$-solvable} if
$D$ is $r$-solvable but the removal of any pebble makes $D$ not
$r$-solvable.  A rooted distribution $D$ is \textit{maximally
$r$-unsolvable} if $D$ is not $r$-solvable but the addition of any
pebble makes $D$ $r$-solvable. A global distribution $D$ is
\textit{minimally solvable} if $D$ is solvable but the removal of
any pebble makes $D$ unsolvable.  A global distribution $D$ is
\textit{maximally unsolvable} if $D$ is unsolvable but the
addition of any pebble makes $D$ solvable.

For example, the first distribution in Figure \ref{K23} is
minimally solvable as a global distribution, since the deletion of
any pebble makes it unsolvable to some root.  However, it is not
minimally $r$-solvable for any choice of a root $r$, since once a
root is selected four pebbles can be deleted from the distribution
while keeping it solvable.

\begin{figure}[ht!]
\begin{center}
\includegraphics[width=2in]{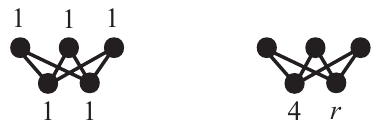}

\caption{Two pebbling distributions on the graph $K_{2,3}$. The first distribution is minimally solvable, but not minimally $r$-solvable for any choice of a root $r$.  The second distribution is both minimally solvable and minimally $r$-solvable.}
\label{K23}
\end{center}
\end{figure}

The eight combinations of largest or smallest, solvable or
unsolvable, and rooted or global distributions yield the following
five pebbling-related parameters on a graph. Two pairs of these
combinations yield the same parameter, and one combination turns
out to be trivial, as noted below.  Given any two of the remaining parameters, there is a graph for which they differ,
as shown in Table 1.

\begin{center}
\begin{tabular}{|c|c|c|c|}
  \hline
   & $K_5$ & $K_{2,3}$ & $C_7$ \\
  \hline
  $p(G)$ & 5 & 5 & 11 \\
  $c_{g}(G)$ & 5 & 5 & 10 \\ \hline
  $c_{r}(G)$ & 2 & 4 & 10 \\
  $2^d$ & 2 & 4 & 8 \\ \hline
  $n$ & 5 & 5 & 7 \\
  $c_{u}(G)$ & 5 & 4 & 7 \\ \hline
  $o(G)$ & 2 & 3 & 5 \\
  \hline
\end{tabular}
\medskip

\begin{small}Table 1. Some pebbling numbers.  New notation is defined below; the parameters are listed in decreasing order based on Figure~\ref{parameter_lattice}. \end{small}
\end{center}

\begin{itemize}

\item The \textit{pebbling number} $p(G)$ is one greater than
    the largest size of a maximally unsolvable global
    distribution on $G$. Equivalently, $p(G)$ is one greater
    than the largest size of a maximally $r$-unsolvable rooted
    distribution on $G$ for any $r$.

For example, the first distribution in Figure \ref{C7} is
maximally unsolvable, since the addition of any one pebble
results in a solvable distribution.  Since there are no
maximally unsolvable distributions on $C_7$ with 11 pebbles,
$p(C_7)=11$.

\item The \textit{$gu$-critical pebbling number} $c_{gu}(G)$
    is one greater than the smallest size of a maximally
    unsolvable global distribution on $G$.  The
    \textit{$ru$-critical pebbling number} $c_{ru}(G)$ is one
    greater than the smallest size of a maximally
    $r$-unsolvable rooted distribution on $G$ for any $r$.
    These two parameters are equal, as proven in Lemma
    \ref{unsolvable_equivalence}.  Consequently, we define
    $c_u(G)=c_{ru}(G)=c_{gu}(G)$.

For example, the fourth distribution in Figure \ref{C7} is
maximally unsolvable, and there are no maximally unsolvable
distributions on $C_7$ with fewer than 6 pebbles.  So
$c_u(C_7)=7$.

\item The \textit{$g$-critical pebbling number} $c_{g}(G)$ is
    the largest size of a minimally solvable global
    distribution on $G$.

For example, the second distribution in Figure \ref{C7} is
minimally solvable, since the deletion of any pebble makes it
unsolvable.  Since there are no minimally solvable
distributions on $C_7$ with greater than 10 pebbles,
$c_g(C_7)=10$.  In particular, the third distribution in
Figure \ref{C7} is not minimally solvable, since removing a
pebble from the vertex with 5 pebbles results in a solvable
distribution.

\item The \textit{$r$-critical pebbling number} $c_{r}(G)$ is
    the largest size of a minimally $r$-solvable rooted
    distribution on $G$ for any $r$.  If a minimally
    $r$-solvable rooted distribution on $G$ has $c_r(G)$
    pebbles, then we call it an \textit{$r$-ceiling}
    distribution.

For example, the second distribution in Figure \ref{K23} is
minimally $r$-solvable.  Since there is no minimally
$r$-solvable rooted distribution on $K_{2,3}$ with greater
than 4 pebbles, $c_r(K_{2,3})=4$.  In particular, as discussed
above, the first distribution in Figure \ref{K23} is not
minimally $r$-solvable for any $r$.

\item The \textit{optimal pebbling number} $o(G)$ is the
    smallest size of a minimally solvable global distribution
    on $G$ \cite{Pachter95}.

For example, the fifth distribution in Figure \ref{C7} is
minimally solvable, and there are no minimally solvable
distributions on $C_7$ with fewer than 5 pebbles.  So
$o(C_7)=5$.

\item The smallest size of a minimally $r$-solvable
    distribution on $G$ is 1 for any connected graph $G$, so we do not
    consider it.

\end{itemize}

Just as the classical pebbling number can be used to model the distribution of resources with loss, critical pebbling numbers might be used to model the spread of harmful substances with loss.  For example, suppose the pebbles represent radioactive waste and the root vertex represents a city.  Our goal is to place the waste far enough away from the city so the residents don't feel the harmful effects.  The largest size of an unsolvable distribution rooted at $r$, which is one less than the $r$-critical pebbling number of the associated graph, is a measure of the largest quantity of waste we can dispose of without harm, and the maximum $r$-insufficient distribution itself is a map of where to dump it.

As discussed above, these pebbling parameters can also be defined in terms of pebbling games on $G$.  Here are the rules for the pebbling games which yield the $g$-critical, $u$-critical, and $r$-critical pebbling numbers.
\bigskip

\noindent \textbf{The $g$-critical Pebbling Game:}

\textbf{Round 1.}  First we place some pebbles on the vertices of $G$.  Then the opponent picks a root vertex.  Then we reach the root with a sequence of pebbling moves.

\textbf{Round 2.} The opponent resets the original distribution, and removes any one pebble.  Then we pick a (possibly different) root vertex.  Then the opponent fails to reach the root with a sequence of pebbling moves.

We win the $g$-critical pebbling game if we win both rounds of the game.  The $g$-critical pebbling number $c_g(G)$ is the largest number of pebbles with which we can always win the $g$-critical pebbling game.
\bigskip

\noindent \textbf{The $r$-critical Pebbling Game:}

\textbf{Round 1.}  First the  opponent picks a root vertex.  Then we place some pebbles on the vertices of $G$.  Then we reach the root with a sequence of pebbling moves.

\textbf{Round 2.} The opponent resets the original distribution, and removes any one pebble.  Then the opponent fails to reach the root with a sequence of pebbling moves.

We win the $r$-critical pebbling game if we win both rounds of the game.  The $r$-critical pebbling number $c_r(G)$ is the largest number of pebbles with which we can always win the $r$-critical pebbling game.
\bigskip

\noindent \textbf{The $u$-critical Pebbling Game:}

\textbf{Round 1.}  First we place some pebbles on the vertices of $G$, and pick a root vertex.  Then the  opponent fails to reach the root with a sequence of pebbling moves.

\textbf{Round 2.} The opponent resets the original distribution, picks a root vertex, and adds one pebble on any vertex. Then we reach the root with a sequence of pebbling moves.

We win the $u$-critical pebbling game if we win both rounds of the game.  The $u$-critical pebbling number is the smallest number of pebbles, counting the additional pebble added in round 2, with which we can always win the $u$-critical pebbling game.
\bigskip

The following is another helpful way of thinking of these pebbling
parameters.  Consider the set of all global distributions on a
given graph $G$. Given the distributions $D$ and $E$, we say that
$D \leq E$ if $D(a) \leq E(a)$ for all vertices $a$ in $G$. With
this ordering, the set of all distributions on $G$ becomes a
lattice. Also note that if $D \leq E$ then $|D| \leq |E|$.

Now divide the distributions in this lattice into the subset $S$
of solvable distributions and the subset $U$ of unsolvable
distributions, and also consider the set $M$ of maximal unsolvable
distributions and the set $m$ of minimal solvable distributions.
The pebbling number, $g$-critical pebbling number, $u$-critical
pebbling number, and optimal pebbling number can be viewed as
maxima and minima of these subsets. Specifically: \begin{itemize}

\item The pebbling number is one greater than the largest size of
    any distribution in $M$.

\item The $u$-critical pebbling number is one greater than the
    smallest size of any distribution in $M$.

\item The $g$-critical pebbling number is the largest size of any
    distribution in $m$.

\item The optimal pebbling number is the smallest size of any
    distribution in $m$.
\end{itemize}
Said another way, among maxima of unsolvable distributions, $p(G)-1$ is the largest and $c_u(G)-1$ is the smallest.  Among minima of solvable distributions, $c_g(G)$ is the largest and $o(G)$ is the smallest.
%
%

The $r$-, $g$-, and $u$-critical pebbling numbers have not been
previously studied.  We now investigate the relationships between
the five distinct pebbling numbers defined above.

\begin{lemma}For any graph $G$, $c_{ru}(G)=c_{gu}(G)$.
\label{unsolvable_equivalence}
\end{lemma}

\begin{proof}
By definition, every maximally unsolvable distribution on $G$ is
maximally $r$-unsolvable for some $r$.

Let $r$ be a vertex of $G$.  It suffices to show that every
maximally $r$-unsolvable distribution is maximally unsolvable. Let
$D$ be a maximally $r$-unsolvable distribution on $G$, and assume
by way of contradiction that $D$ is not maximally unsolvable.
Since $D$ is $r$-unsolvable, $D$ is unsolvable.  Since $D$ is not
maximally unsolvable, there exists a vertex $s$ of $G$ such that
$D$ is $s$-unsolvable, and a pebble can be added to $D$ so that
$D$ remains $s$-unsolvable.  Hence, $s \not= r$.

Consider the distribution $E$ obtained from $D$ by adding a pebble
on $s$.  Since $D$ is $s$-unsolvable, no second pebble from $E$
can be moved to $s$.  So any solution of $E$ to $r$ does not use
the pebble on $s$.  But this means that $E$ is not $r$-solvable,
which contradicts the fact that $D$ is maximally $r$-unsolvable.
\end{proof}

\begin{lemma}
Suppose that $G$ is a graph with $n$ vertices and diameter $d$.
Then

\begin{enumerate}[(i)]

\item $o(G) \leq 2^d$,

\item $2^d \leq c_r(G)$,

\item $c_r(G) \leq c_g(G)$,

\item $o(G) \leq c_u(G)$,

\item $c_u(G) \leq n$,

\item $n \leq c_g(G)$, and

\item $c_g(G) \leq p(G).$
\end{enumerate}

\label{optimal_lemma} \label{upper_bound} \label{lower_bound}
\end{lemma}

{\textbf{Proof.}}
\begin{enumerate}[(i)]

\item The proof appears in \cite{Pachter95}.

\item Let $a$ and $b$ be two vertices which are distance $d$
    apart. Let $D$ be the rooted distribution with $a=r$ and
    $D(b)=2^d$. This distribution is minimally $r$-solvable,
    and consequently $2^d \leq c_r(G)$.

\item \noindent \textbf{Case 1: There exists an $r$-ceiling
    distribution $D$ on $G$ which is not solvable.} Suppose
    $D$ is not solvable to the vertex $s$. Consider the
    distribution $D_1$ obtained from $D$ by adding a pebble at
    $s$. Since $D$ is not solvable to $s$, the new pebble
    cannot be used in a solution of $D_1$ to $r$. Hence if we
    remove any pebble in $D_1$, either $D_1$ is no longer
    solvable to $r$, or $D_1$ is no longer solvable to $s$.
    Either $D_1$ is solvable, or $D_1$ is not solvable to some
    vertex $t$. In this second case we form the distribution
    $D_2$ by adding a pebble at $t$ to $D_1$. Continuing in
    this way, we eventually arrive at a minimally solvable
    distribution $E$ which contains all the pebbles in D and
    some additional pebbles.  Since $c_r(G)=|D|$ and $c_g(G)
    \geq |E|$, $c_r(G) \leq c_g(G)$.

\noindent \textbf{Case 2: All $r$-ceiling distributions on $G$
are solvable.}  Every $r$-ceiling distribution is minimally
$r$-solvable, and hence minimally solvable.  Since $c_r(G)$ is
the maximum size of an $r$-ceiling distribution and $c_g(G)$
is the maximum size of a minimally solvable distribution,
$c_r(G) \leq c_g(G)$.

\item Because $c_u(G)$ is one larger than the size of a
    maximally unsolvable distribution, there exists a solvable
    distribution with $c_u(G)$ pebbles.  Since $o(G)$ is the
    size of the smallest solvable distribution on $G$, $o(G)
    \leq c_u(G)$.

\item Any distribution on the graph $G$ with one pebble on all
    but one vertex is maximally unsolvable. Since $c_u(G)$ is
    one greater than the smallest such distribution, $c_u(G)
    \leq n$.

\item The distribution with one pebble on every vertex of $G$
    has $n$ pebbles and is minimally solvable.  Since $c_g(G)$
    is the size of the largest such distribution, $n \leq
    c_g(G)$.

\item By the definition of $p(G)$, every distribution with
    $p(G)$ or more pebbles is solvable.  By the definition of
    $c_g(G)$, there exist distributions with $c_g(G)-1$
    pebbles which are not solvable. \hfill $\square$
\end{enumerate}

Table 1 shows the values of the five pebbling parameters for three
graphs. The table illustrates that the inequalities in Lemma
\ref{upper_bound} can be either strict or not, and that there is
no definite relationship between the remaining pairs of pebbling
parameters.   For instance, $c_r(K_5)<c_u(K_5)$, but
$c_r(C_7)>c_u(C_7)$.  The proof that $p(C_7)=11$ appears in
\cite{Hurlbert99},
 and we prove $c_r(C_7)=10$ in Lemma \ref{cycle} below.  Since the remainder of the
paper focuses on $c_r(G)$, we leave the rest of the values for the
reader to verify.

We summarize the relationships between these seven values in the
lattice in Figure \ref{parameter_lattice}.  In this figure, an
edge indicates that the lower value is less than
or equal to the upper value for all graphs, and a missing edge
indicates that each value may be greater than the other on some
graphs.  Finally, note that for the graph with a single vertex,
all seven values are equal.


\begin{figure}[ht!]
\begin{center}
\includegraphics[width=1.5in]{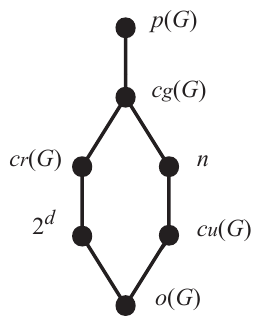}

\caption{The relationships between the five pebbling parameters.}
\label{parameter_lattice}
\end{center}
\end{figure}

\section{The $r$-Critical Pebbling Number}

For the remainder of the paper, we focus on the $r$-critical pebbling number.  We say that a minimally
$r$-solvable rooted distribution is an \textit{$r$-critical
distribution}, so $c_r(G)$ is the maximum size of an $r$-critical
distribution on $G$.  Recall that the $r$-critical distributions
on $G$ with $c_r(G)$ pebbles are called \textit{$r$-ceiling
distributions}.

We say that the rooted distribution $D$ is \textit{$r$-excessive}
if $D$ is $r$-solvable and not $r$-critical, and
\textit{$r$-insufficient} if $D$ is not $r$-solvable. Then the
sets of $r$-insufficient, $r$-critical, and $r$-excessive
distributions on $G$ form a partition of all rooted distributions
on $G$. Note that for an $r$-insufficient distribution $I$, an
$r$-critical distribution $C$, and an $r$-excessive distribution
$E$, we may have $|E|<|C|<|I|$. Examples of three such
distributions are shown in Figure \ref{distributions}.

\begin{figure}[ht!]
\begin{center}
\includegraphics[width=2.5in]{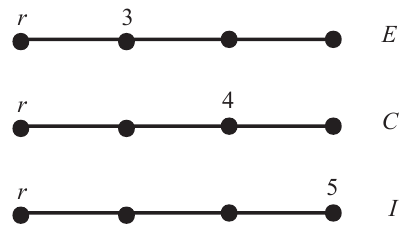}

\caption{Three rooted distributions on the graph $P_4$. From top to bottom, they are $r$-excessive, $r$-critical, and $r$-insufficient.}
\label{distributions}
\end{center}
\end{figure}

Given a distribution $D$ with root $r$, we say that a solution of $D$ is an \textit{$r$-critical solution} if it leaves one pebble
on $r$ and no pebble on any other vertex.

\begin{lemma}A rooted distribution $D$ is $r$-critical if and only if $D$ is $r$-solvable and
all solutions of $D$ are $r$-critical. \label{critical_solution}
\end{lemma}

\begin{proof}
Suppose $D$ is $r$-critical.  Then by definition $D$ is
$r$-solvable. Suppose there exists a solution $S$ of $D$ which is
not $r$-critical. Then $S$ leaves a pebble on the non-root vertex
$a$. If this pebble is unmoved from $D$, we may delete it from $D$
and obtain an $r$-solvable rooted distribution, which contradicts
the fact that $D$ is $r$-critical.  So $S$ must include the
pebbling step $[b,a]$ for some other vertex $b$. Again, if a
pebble on $b$ is unmoved from $D$ until this pebbling step, then
$D$ would not be $r$-critical. We continue in this way. Since $S$
is a finite sequence of pebbling steps, eventually we will find a
pebble in $D$ which may be deleted to obtain an $r$-solvable
rooted distribution. Hence, all solutions of $D$ are $r$-critical.

Conversely, suppose $D$ is $r$-solvable and all solutions of $D$
are $r$-critical.  Let $E$ be a rooted distribution obtained by
removing a pebble from $D$, and suppose that $E$ is $r$-solvable.
Then there exists a solution of $D$ which leaves this pebble
unmoved. So this solution is not $r$-critical, which is a
contradiction. Therefore, $D$ must be $r$-critical.
\end{proof}

\begin{corollary}If $D$ is $r$-critical and $a$ is a vertex of $G$ with degree 1 distinct from
$r$, then $D(a)$ is even. \label{even} \end{corollary}

\begin{proof}
Suppose $D(a)$ is odd, and $S$ is a solution of $D$.  If $[b,a]$
is a pebbling step in $S$, then $[a,b]$ must also be a pebbling
step in $S$. Removing both pebbling steps from $S$ results in a
non-$r$-critical solution of $D$, so $D$ is not $r$-critical by
Lemma \ref{critical_solution}. Alternatively, if there are no
pebbling steps in $S$ of the form $[b,a]$, then $S$ leaves at
least one pebble on $a$.  So $S$ is not $r$-critical, and again by
Lemma \ref{critical_solution}, $D$ is not $r$-critical.
\end{proof}


\begin{lemma}
Suppose that $G$ is a graph and $a$ is a vertex of $G$ which is
adjacent to every other vertex of $G$. Then the $r$-critical
distributions on $G$ must have one of the following forms:

\begin{enumerate}[(i)]

\item One pebble on $r$, and no pebble on any other vertex.

\item Two pebbles on $a$, and no pebble on any other vertex, including $r$.

\item Four pebbles on some vertex $b$, and no pebble on any other vertex, including $a$ and $r$.

\item Two pebbles on two vertices $b$ and $c$, and no pebble on any other vertex, including $a$ and $r$.

\item Two pebbles on some vertex $b$, no pebbles on $r$, and less than
    two pebbles on all other vertices, including $a$.

\end{enumerate} \label{cases}

\end{lemma}

\begin{proof}
Suppose $D$ is an $r$-critical distribution on $G$, and $b$ and
$c$ are vertices in $G$ other than $a$. If $D$ has either one
pebble on $r$, or two pebbles on $a$, or four pebbles on $b$, or
two pebbles on both $b$ and $c$, then $D$ must have no pebbles on
any other vertex. Alternatively, suppose that none of these
conditions are met. In this case, $D$ has no pebbles on $r$, less
than two pebbles on $a$, less than four pebbles on any other
vertex, and more than one pebble on at most one vertex.

If $D$ has more than one pebble on no vertex, then $D$ is
$r$-insufficient.  So without loss of generality, $2 \leq D(b) <4$
and $D(v)<2$ for all other vertices $v$ in $G$.  If $D(b)=3$ then
any solution of $D$ will leave at least one pebble on $b$, and $D$
will not be $r$-critical. So $D$ must have form (v).  We have shown
that all $r$-critical distributions on $G$ must have one of the
five forms listed.
\end{proof}

\begin{theorem}The star $K_{1,n}$ has pebbling number $n+2$ and
$r$-critical pebbling number $4$ for $n \geq 4$.
\label{star_theorem}
\end{theorem}

\begin{proof}
The fact that $p(K_{1,n})=n+2$ is a corollary of Theorem 4 of
\cite{Moews92}, which gives a formula for the pebbling number of
any tree.

By Lemma \ref{lower_bound}, $c_r(K_{1,n}) \geq 4$. By Lemma
\ref{cases}, the only $r$-critical distributions on $K_{1,n}$ with
more than four pebbles must have two pebbles on one vertex, and
one pebble on at least three other vertices.  But by 
Corollary~\ref{even}, there is only one vertex in $K_{1,n}$ which can have
one pebble in an $r$-critical distribution. Hence
$c_r(K_{1,n})=4$.
\end{proof}

Note that Lemma \ref{upper_bound} gives us $c_r(G) \leq p(G)$ for
any graph $G$.  But in fact, Theorem \ref{star_theorem} gives an
example of a family of graphs for which the difference
$p(G)-c_r(G)$ is arbitrarily large.

Recall the \textit{fan} $F_k$, the path $P_k$ on $k$ vertices
with an additional vertex $x$ adjacent to every vertex in $P_k$.
$F_8$ is shown in Figure \ref{broken_wheel}.  The following theorem appears in~\cite{Feng02}.

\begin{figure}[ht!]
\begin{center}
\includegraphics[width=\textwidth]{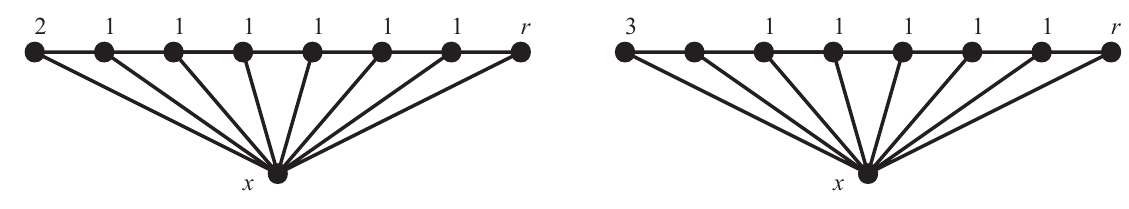}

\caption{An $r$-ceiling distribution and an $r$-insufficient
distribution on the fan $F_8$.} \label{broken_wheel}
\end{center}
\end{figure}

\begin{theorem} The fan $F_k$ has pebbling number $k+1$ for
$k \geq 4$.\label{broken_wheel_p}
\end{theorem}

%

\begin{theorem}The fan $F_k$ has $r$-critical pebbling number $k$ for $k \geq 4$.
\label{wheel_theorem} \end{theorem}

\begin{proof}
We first prove that $c_r(F_k) \leq k$. Let $D$ be an $r$-critical
distribution on $F_k$ with more than $k$ pebbles.  Since $k \geq
4$, by Lemma \ref{cases}, $D$ must have two pebbles on one vertex
and one pebble on at least $k-1$ other vertices.  Therefore,
either $r$ has a pebble on it, or $r=x$, or $x$ has a pebble on
it. In each of these cases, one can easily check that $D$ is
$r$-excessive.  Thus, there are no $r$-critical distributions on
$F_k$ with more than $k$ pebbles.

The first rooted distribution shown in Figure \ref{broken_wheel}
is $r$-critical as long as $k \geq 4$, and has $k$ pebbles.  It
follows that $c_r(F_k)=k$ for $k \geq 4$.
\end{proof}

\begin{corollary}
For a positive integer $k$, there exist graphs with $r$-critical pebbling number $k$ if and only if $k \not=3$.
\end{corollary}

\begin{proof}
For $k=1,2$, the path $P_k$ has $r$-critical pebbling number $k$.
For $k \geq 4$, the fan $F_k$ has $r$-critical pebbling number $k$
by Theorem \ref{wheel_theorem}.

Now suppose that $D$ is a rooted distribution on the graph $G$
with 3 pebbles.  We show that $D$ is not an $r$-ceiling
distribution. If $r$ has a pebble on it then $D$ is $r$-excessive.  Otherwise, if there is no vertex $a$ with $D(a) \geq 2$, then
there are no legal pebbling steps from $D$, and $D$ is
$r$-insufficient.  If there is such a vertex, and $a$ is adjacent
to every other vertex of $G$, then two pebbles from $a$ can be
used to pebble to $r$, and $D$ is $r$-excessive.  If $a$ is not
adjacent to every other vertex of $G$, then $d(G) \geq 2$, so
$c_r(G) \geq 4$. Hence, there is no $r$-ceiling distribution with
three pebbles.
\end{proof}

Also note that $F_k$ is an example of a graph with diameter 2 and
arbitrarily large $r$-critical pebbling number.  Thus, $c_r(G)$ is
not bounded above by any function of $d(G)$.  

\section{Greed, Thrift, and Weight}

We denote the distance between the vertices $a$ and $b$ by
$d(a,b)$.  The pebbling step $[a,b]$ is \textit{greedy} if
$d(a,r)>d(b,r)$; in other words, the step moves towards the root.
The rooted distribution $D$ is \textit{greedy} if there is a
solution of $D$ which uses only greedy pebbling steps.  The graph
$G$ is \textit{greedy} if every distribution with at least $p(G)$
pebbles is greedy \cite{Hurlbert99}.  The graph $G$ is
\textit{thrifty} if every $r$-critical distribution with at least
$c_r(G)$ pebbles (i.e. every $r$-ceiling distribution) is greedy.

In general, computing the pebbling number of a graph is a hard problem.  However, Hurlbert notes that for greedy graphs, the
computation becomes much easier \cite{Hurlbert99}. For thrifty
graphs, the $r$-critical pebbling number is even easier to
compute. As we shall see below, the $r$-critical pebbling number
of a thrifty graph is determined by its diameter.

The \textit{weight} of the rooted distribution $D$ is the value
\begin{equation*} \label{sum} w(D)=\sum_{v \in V(G)}
\frac{D(v)}{2^{d(v,r)}}.\end{equation*} The \textit{weight} $w(G)$
of $G$ is the largest weight of any $r$-ceiling distribution on
$G$. Note that there may be $r$-critical distributions with fewer
pebbles and larger weight, for example, given any rooted graph
$G$, the distribution with two pebbles on one vertex adjacent to
$r$ has weight 1.  We are interested only in the weight of
$r$-critical distributions with exactly $c_r(G)$ pebbles.

\begin{lemma}If the rooted distribution $E$ is obtained from the
rooted distribution $D$ by a greedy pebbling step, then
$w(E)=w(D)$.  If the rooted distribution $E$ is obtained from the
rooted distribution $D$ by a non-greedy pebbling step, then
$w(E)<w(D)$. \label{step_weight}
\end{lemma}

\begin{proof}
Suppose $E$ is obtained from $D$ by the pebbling step $[a,b]$.  If
$[a,b]$ is greedy and $d(a,r)=s$, then $d(b,r)=s-1$.  $E$ has two
fewer pebbles on $a$ and one additional pebble on $b$.  So
$w(E)=w(D)-\frac{2}{2^s}+\frac{1}{2^{s-1}}=w(D)$.

If $[a,b]$ is not greedy and $d(a,r)=s$, then $d(b,r)=t \geq s$.
So $w(E)=w(D)-\frac{2}{2^s}+\frac{1}{2^t}<w(D)$.
\end{proof}

\begin{lemma}If $D$ is $r$-critical and greedy then $w(D)=1$. \label{dist_weight}
\end{lemma}

\begin{proof}
We proceed by induction.  The only $r$-critical distribution on
$G$ with size 1 is the rooted distribution with one pebble on $r$
and no pebble on any other vertex.  This rooted distribution is
$r$-critical, greedy, and has weight 1 on every graph $G$.

Suppose that $D$ is an $r$-critical, greedy distribution on $G$
with $k$ pebbles. Since $D$ is greedy, there exists a greedy
pebbling step $[a,b]$ from $D$ to a new rooted distribution $E$
with $k-1$ pebbles which is the first pebbling step in a greedy solution
of $D$.

If $E$ is not $r$-critical, then by Lemma \ref{critical_solution}
there exists a non-$r$-critical solution of $E$.  Appending the
pebbling step $[a,b]$ to this solution yields a non-$r$-critical
solution of $D$.  Therefore, $E$ is $r$-critical.  Because $[a,b]$
is the first step in a greedy solution of $D$, the remainder of
this solution is a greedy solution of $E$, so $E$ is also greedy.
By the induction hypothesis $w(E)=1$, and by Lemma
\ref{step_weight}, $w(D)=w(E)$. Consequently, $w(D)=1$.
\end{proof}

\begin{corollary}If $w(D)<1$ then $D$ is $r$-insufficient. \label{min_weight}
\end{corollary}

\begin{proof}
By Lemma \ref{step_weight}, no pebbling step increases the weight
of a rooted distribution, and the weight of a solved rooted
distribution is at least 1.  It follows that the weight of any
$r$-solvable rooted distribution is at least 1.
\end{proof}

\begin{corollary}For any graph $G$, $w(G) \geq 1$. \label{weight_bound} \hfill $\square$
\end{corollary}

\begin{theorem}The graph $G$ is thrifty if and only if $w(G)=1$.
\label{thrifty_weight}
\end{theorem}

\begin{proof}
Suppose $G$ is thrifty and $D$ is an $r$-ceiling distribution on
$G$. Then since $G$ is thrifty, $D$ is greedy.  By Lemma
\ref{dist_weight}, $D$ has weight 1.  As this is true for any
$r$-ceiling distribution on $G$, $w(G)=1$.

Conversely, suppose $w(G)=1$ and $D$ is an $r$-ceiling
distribution on $G$. Then $D$ has weight 1.  Suppose by way of contradiction that there is a solution of $D$ with a non-greedy pebbling step, say from distribution $D'$ to distribution $D''$ on $G$.  Then by Lemma~\ref{step_weight}, $w(D'')<w(D')$.  Also by Lemma~\ref{step_weight}, no pebbling step increases the weight of a distribution.  Then the solution of $D$ ends with a weight of less than 1.  This contradicts Corollary~\ref{min_weight}.  
So $D$ must be $r$-solvable using only greedy pebbling steps.  Thus, $D$ is
greedy, and $G$ is thrifty.
\end{proof}

\begin{theorem}If $G$ is a thrifty graph with diameter $d$, then
$c_r(G)=2^d$. \label{thrifty_number}
\end{theorem}

\begin{proof}
Suppose $D$ is an $r$-ceiling distribution on $G$.  By
Lemma
\ref{lower_bound}, $|D| \geq 2^d$.  Since $G$ is thrifty, by Theorem
\ref{thrifty_weight}, $w(D)=1$. Because every pebble in $D$ must
contribute at least $\frac{1}{2^d}$ to $w(D)$, there must be
exactly $2^d$ of them.
\end{proof}

By Theorems \ref{thrifty_weight} and \ref{thrifty_number}, all
thrifty graphs achieve the lower bound for weight given in
Corollary~\ref{min_weight} and the lower bound for $r$-critical
pebbling number given in Lemma~\ref{lower_bound}. Thus, thrifty
graphs are in some sense the simplest graphs with respect to
$r$-critical pebbling number. However, although all graphs with
weight 1 are thrifty, we prove in Theorem \ref{goggles_theorem}
that not all graphs with $r$-critical pebbling number $2^d$ are
thrifty.

\section{Separating Examples}

The goal of this section is to provide examples of graphs in all of the regions of the Venn diagram in Figure~\ref{pebbling_classes}.  We consider five specific graphs, which we prove distinguish
the classes of greedy graphs, thrifty graphs, and graphs $G$ for
which $c_r(G)=2^d$.  The first of these is $C_7$, the cycle on 7
vertices.  The remaining four graphs we call $G_1$ through $G_4$,
and display them in Figures \ref{glasses}, \ref{goggles},
\ref{sting_ray}, and \ref{romulan}, respectively. For each graph,
we first determine its pebbling number and $r$-critical pebbling
number, and then prove that it has the required properties to fit in the claimed region of the Venn diagram in Figure~\ref{pebbling_classes}.

To determine $c_r(C_7)$, it will be useful to have the following
lemma.  If $H$ is a subgraph of $G$, $D$ is a rooted distribution
on $G$, and $E$ is a rooted distribution on $H$, we say that $E$
is \textit{induced} from $D$ if the root of $E$ is the root of
$D$, and $E(a)=D(a)$ for all vertices $a$ of $H$.

\begin{lemma}
Suppose that $D$ is a rooted distribution on the graph $G$, $P$ is a path
in $G$ with end vertex $r$, and $E$ is the rooted distribution on the subgraph $P$
induced from $D$.  If $w(E)>1$ on $P$, then $D$ is $r$-excessive on $G$.
\label{path_lemma}
\end{lemma}

\begin{proof}
Suppose $w(E)>1$ on the graph $P$.  If there is a pebble on $r$, then there is also a pebble not on $r$.  Then $D$ is $r$-excessive on $G$.  Now suppose there is no pebble on $r$.  If there is at most one pebble on every other vertex of $P$, then $w(E)<1$.  Therefore there must be more than one pebble on some vertex $a$ of $P$.  We pebble from $a$ towards $r$ on $P$.
Since this pebbling step is greedy, by Lemma \ref{step_weight},
the new rooted distribution $E'$ obtained from this pebbling step
still satisfies $w(E')>1$. We may continue in this way until we
reach a rooted distribution $F$ on $P$ with a pebble on $r$.  Because $w(F)>1$, again $F$ has a pebble not on $r$, so $F$ is
$r$-excessive.  As we may use these same
pebbling steps on $D$, $D$ is also $r$-excessive.
\end{proof}

\begin{theorem}The cycle $C_7$ is not thrifty, not greedy, and has $r$-critical pebbling
number greater than $2^d$. \label{cycle}
\end{theorem}

\begin{proof}
The pebbling number of $C_7$ is 11 \cite{Pachter95}.  Let $r$ be
any vertex of $C_7$, and let $u$ and $v$ be the two vertices
farthest from $r$.  Then it is easy to verify that the rooted
distribution $D(u,v)=(5,6)$ is not greedy.

Suppose that $D$ is a rooted distribution on $C_7$ with 11 or more
pebbles, and suppose that the number of pebbles in $D$ on each of
the six non-root vertices of $C_7$ are $a$, $b$, $c$, $d$, $e$,
and $f$, starting from a vertex adjacent to $r$ and continuing
around the cycle.  We consider the two paths from $r$ clockwise
and counterclockwise around the cycle.  If $D$ is $r$-critical,
then by Lemma \ref{path_lemma}, $D$ must satisfy
\begin{align*}\frac{a}{2}+\frac{b}{4}+\frac{c}{8}+\frac{d}{16}+\frac{e}{32}+\frac{f}{64}
\leq 1 \text{ and} \\
\frac{a}{64}+\frac{b}{32}+\frac{c}{16}+\frac{d}{8}+\frac{e}{4}+\frac{f}{2}
\leq 1.
\end{align*}
Adding these inequalities and simplifying yields
\begin{align*}33a+18b+12c+12d+18e+33f \leq 128 \\
12(a+b+c+d+e+f)+21a+6b+6e+21f \leq 128.
\end{align*}
Since $|D|=a+b+c+d+e+f \geq 11$, we have \begin{align*} 132=12
\cdot 11 \leq 12(a+b+c+d+e+f)+21a+6b+6e+21f \leq 128, \end{align*}
which is a contradiction.  Consequently, there is no $r$-critical
distribution on $C_7$ with 11 or more pebbles, and $c_r(C_7) \leq
10$.  As the rooted distribution $D(u,v)=(4,6)$ is $r$-critical,
$c_r(C_7)=10$.  Also, this rooted distribution is not greedy, so
$C_7$ is not thrifty.

Finally, $d(C_7)=3$, and so $c_r(C_7)=10>8=2^d$.
\end{proof}
\bigskip

\begin{figure}[ht!]
\begin{center}
\includegraphics[width=\textwidth]{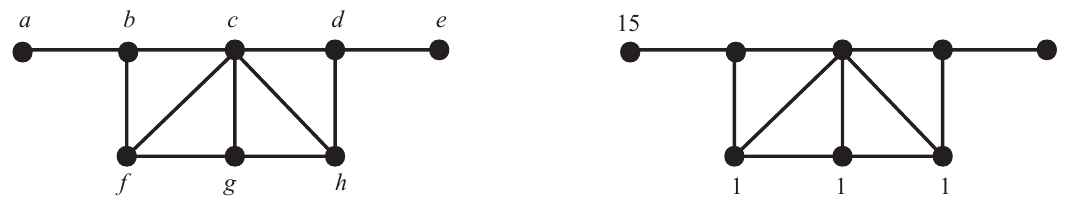}

\caption{A labeling and a rooted distribution on the graph $G_1$.}
\label{glasses}
\end{center}
\end{figure}

%

\begin{lemma}The pebbling number of $G_1$ is 18. \label{glasses_pebbling}
\end{lemma}

\begin{proof}
We follow the labeling system given in Figure \ref{glasses}.  Let
$H$ be the induced subgraph of $G_1$ on the vertices $b$, $c$,
$d$, $f$, $g$, and $h$. Note that $H \cong F_5$.  We will show
that any distribution on $G_1$ with 18 pebbles is solvable.

Let $D$ be a distribution on $G_1$ with 18 pebbles.   Assume without loss of generality that $D(a) \geq D(e)$.  If there are no pebbles on $a$, then there are also no pebbles on $e$, so we may assume that $a$ is not the root of $D$.  If there are pebbles on $a$, we may also assume that $a$ is not the root of $D$. So in all cases we may assume that $a$ is not the root.  Therefore in any solution of $D$, we will never make a pebbling step from $b$ to $a$, only from $a$ to $b$.  So we may also assume that $D(a)$ is odd because, 
if $D(a)$ is even, we may add one pebble at $a$ without affecting the solvability of $D$.
There are five cases to consider:

\noindent \textbf{Case 1. The vertex $r$ is in $H$.} 
We have 18
pebbles in our distribution.  We move as many pebbles as possible
from $a$ and $e$ into $H$.  After this, there is at most 1 pebble remaining on each
of $a$ and $e$.  At most half of the other 16 pebbles may be used in these pebbling steps, so there are at least 8 pebbles
on $H$ afterwards.  But $p(H)=6$ by Theorem \ref{broken_wheel}, so the resulting distribution of at least 8 pebbles on $H$ is solvable.

\noindent \textbf{Case 2. $D(a) \geq 17.$} In this case, we may
pebble to any root from $a$, since $d(G_1)=4$.

\noindent \textbf{Case 3. $D(a)=15$ and $r=e.$}  In this case, we
perform the pebbling step $[a,b]$ 7 times, yielding at least 7
pebbles on $b$, and 10 pebbles total on $H$. If there is an
additional pebble on $b$, $c$, or $d$, or two pebbles on $f$, $g$,
or $h$, the resulting distribution is solvable.  The
only case left is $D(b,f,g,h)=(7,1,1,1)$.  Then the rooted
distribution is still $r$-solvable.

\noindent \textbf{Case 4. $D(a)=13$ and $r=e.$} 
We perform the
pebbling step $[a,b]$ 6 times, yielding at least 6 pebbles on $b$,
and 11 pebbles total on $H$.  Since $H$ has diameter 2, by \cite{Pachter95} $H$ satisfies the 2-pebbling property, which states that if $H$ has $2p(H)-q+1$ pebbles on $q$ distinct vertices, then two pebbles may be moved to any specified vertex.  Since $p(H)=6$ by Theorem
\ref{broken_wheel_p}, if there are pebbles on at least two vertices of $H$, then starting with 11 pebbles on $H$ we may pebble two pebbles to $d$, and then one to $e$.  If the 11 pebbles are all on $b$ then we may pebble directly to $e$ on the path $bcde$.

\noindent \textbf{Case 5. $D(a) \leq 11$ and $r=e.$}  We perform
the pebbling step $[a,b]$ as many times as possible, yielding at
least 12 pebbles total on $H$.  Since $p(H)=6$ by Theorem
\ref{broken_wheel_p}, this means that we can pebble two pebbles to
$d$, and one to $r$.

So $p(G_1) \leq 18$.  If we remove the pebble on $f$ from the
rooted distribution in Figure \ref{glasses}, then this rooted
distribution is no longer $r$-solvable.  So $p(G_1)=18$.
\end{proof}

\begin{lemma}The $r$-critical pebbling number of $G_1$ is 16. \label{glasses_critical}
\end{lemma}

\begin{proof}
By Lemma \ref{lower_bound}, $c_r(G_1) \geq 16$. We show that any
$r$-critical distribution on $G_1$ has at most 16 pebbles. Again,
we use the labeling system of Figure \ref{glasses}, and identify
$H$ as the induced subgraph of $G_1$ on the vertices
$\{b,c,d,f,g,h\}$.

\noindent \textbf{Case 1. $r$ is a vertex of $H$.}

Let $D$ be a rooted distribution on $G_1$ with 16 pebbles.  We
move as many pebbles as possible from $a$ and $e$ into $H$.  We
may leave at most 1 pebble on each of $a$ and $e$, leaving at
least 14 pebbles total.  At most half of these may be used moving
into $H$, leaving at least 7 pebbles on $H$. Because $c_r(H)=5$ by
Theorem \ref{wheel_theorem}, $D$ is $r$-excessive. Hence, there
are no $r$-critical distributions on $G_1$ with $r$ in $H$.
Without loss of generality, there is only one case remaining.

\noindent \textbf{Case 2. $r=e.$}

By definition, every $r$-critical distribution must have a
solution $S$ which leaves one pebble on $e$ and no pebble on any
other vertex. Let the rooted distributions arrived at after each
pebbling step of $S$ be $D_{15}$, $D_{14},$ $\ldots,$ $D_2$,
$D_1$.  We consider the pebbling steps in $S$ in reverse order.
The last pebbling step in $S$ must be $[d,e]$, from the rooted
distribution $D_2$ with two pebbles on $d$ and no pebble on any
other vertex to the rooted distribution $D_1$ with one pebble on
$e$.

We color one of the pebbles in $D_2$ red and the other blue.  For
each pebbling step before $[d,e]$, we identify the pebble produced
from the pebbling step as either red or blue, and color the two
input pebbles the same color.  Working back through these rooted
distributions in this way, we arrive at a coloring of the pebbles
in $D$.

We consider the rooted distributions $R$ and $B$ of only red and
blue pebbles, respectively.  Note that $\{R,B\}$ is a partition of
the distribution $D$ into not necessarily equal sizes.  The rooted
distributions $R$ and $B$ are each $r$-critical distributions on
$G_1-e$ with root $d$, since any non-$r$-critical solution of one
of these rooted distributions would result in a non-$r$-critical
solution of $D$.

Without loss of generality, a solution of $R$ can begin by
pebbling all of the pebbles on $a$ in $R$ to $b$, obtaining the
rooted distribution $R'$.  However, $R'$ is an $r$-critical
distribution on $H$, so it must have one of the forms given in
Lemma \ref{cases}. This implies that $R$ has at most 8 pebbles,
and the only form of $R$ with eight pebbles has all eight pebbles
on $a$.  This is also true for $B$.

Since $|R| \leq 8$ and $|B| \leq 8$, $|D| \leq 16$.  So the
largest $r$-critical distributions on $G_1$ have 16 pebbles, and
the only $r$-ceiling distributions on $G_1$ have 16 pebbles on
$a$. Consequently, $c_r(G_1)=16$.
\end{proof}

Note that since $d(G_1)=4$, $c_r(G_1)=2^d$.

\begin{theorem}The graph $G_1$ shown in Figure 7 is thrifty but not greedy.
\end{theorem}

\begin{proof}
By Lemma \ref{glasses_critical}, the only $r$-ceiling
distributions on $G_1$ have 16 pebbles at one of the two vertices
$a$ and $e$, and the root at the other vertex.  As these rooted
distributions are greedy, $G_1$ is thrifty.

However, the rooted distribution given on $G_1$ in Figure
\ref{glasses} has $p(G_1)$ pebbles and is not greedy.  Thus, $G_1$
is not greedy.
\end{proof}


\begin{figure}[ht!]
\begin{center}
\includegraphics[width=\textwidth]{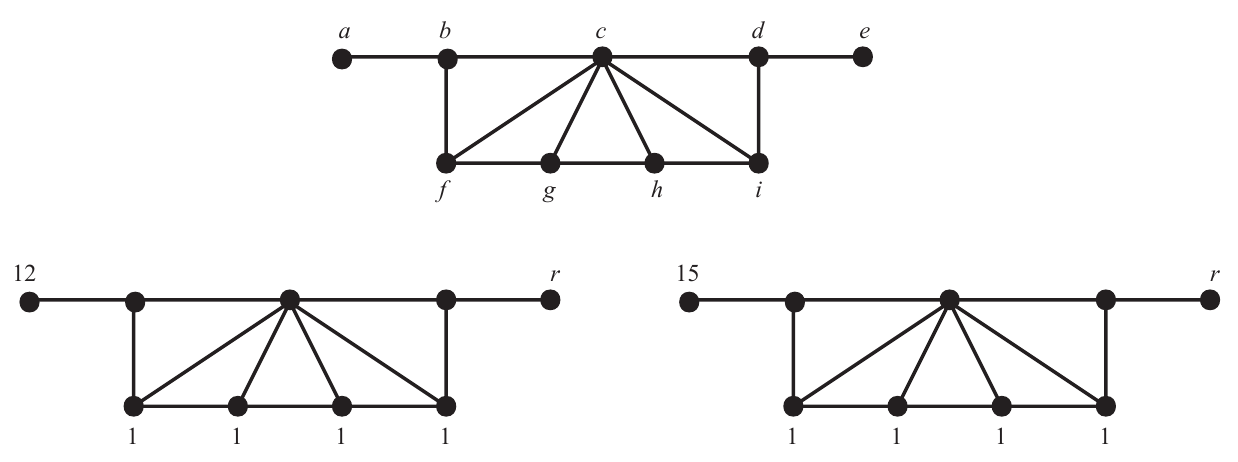}

\caption{A labeling and two non-greedy rooted distributions on the
graph $G_2$.} \label{goggles}
\end{center}
\end{figure}

\begin{lemma} The pebbling number of $G_2$ is 19. \label{pebbling_goggles}
\end{lemma}

\begin{proof}
The proof is analogous to the proof of Lemma~\ref{glasses_pebbling}. We use the labeling system of $G_2$ given
in Figure \ref{goggles}.   As in the proof of Lemma~\ref{glasses_pebbling}, let $H$ be the induced subgraph of $G_1$ on the vertices $b$, $c$, $d$, $f$, $g$, $h$, and $i$. Note that $H \cong F_6$. We verify that every rooted distribution
on $G_2$ with 19 pebbles is $r$-solvable. 

Let $D$ be a distribution on $G_2$ with 19 pebbles.   Assume without loss of generality that $D(a) \geq D(e)$.  As in the proof of Lemma~\ref{glasses_pebbling}, we may assume that $a$ is not the root of $D$ and that $D(a)$ is odd.
There are five cases to consider:

\noindent \textbf{Case 1. The vertex $r$ is in $H$.} 
As in the proof of Lemma~\ref{glasses_pebbling}, moving as many pebbles as possible into $H$ yields at least 8 pebbles on $H$, and since $H$ has pebbling number 7 by Theorem \ref{broken_wheel}, the resulting distribution of at least 8 pebbles on $H$ is solvable.

\noindent \textbf{Case 2. $D(a) \geq 17.$} In this case, we may
pebble to any root from $a$, since $d(G_1)=4$.

\noindent \textbf{Case 3. $D(a)=15$ and $r=e.$}  In this case, we
perform the pebbling step $[a,b]$ 7 times, yielding at least 7
pebbles on $b$, and 11 pebbles total on $H$. If there is an
additional pebble on $b$, $c$, or $d$, or two pebbles on $f$, $g$, $h$, or $i$, the resulting distribution is solvable.  The
only case left is $D(b,f,g,h,i)=(7,1,1,1,1)$.  Then the rooted distribution is still $r$-solvable.

\noindent \textbf{Case 4. $D(a)=13$ and $r=e.$} 
We perform the pebbling step $[a,b]$ 6 times, yielding at least 6 pebbles on $b$,
and 12 pebbles total on $H$.  If at least 8 of these are on the path $bcde$ then the distribution is solvable to $e$.  If 7 pebbles are on this path then one of the vertices $f$, $g$, $h$, or $i$ must have two pebbles.  After the pebbling step from this vertex to $c$, the path $bcde$ has 8 pebbles again.  If $bcde$ has only 6 pebbles and one of the vertices $f$, $g$, $h$, or $i$ has 4 pebbles, we make two pebbling steps from this vertex to $c$, and again $bcde$ has 8 pebbles.  Finally, if none of these are the case, then at least three vertices in $H$ have pebbles on them.  Then since $H$ satisfies the 2-pebbling property, we may pebble two pebbles to $d$, and then one to $e$.

\noindent \textbf{Case 5. $D(a) \leq 11$ and $r=e.$}  We perform
the pebbling step $[a,b]$ as many times as possible, yielding at
least 13 pebbles total on $H$.  If all 13 pebbles are on any vertex, we may pebble from this vertex to $e$.  Otherwise, again by the 2-pebbling property, we can pebble two pebbles to
$d$, and one to $e$.
%

So $p(G_2) \leq 19$.  If we remove the pebble on $f$ from the
second rooted distribution in Figure \ref{goggles}, the resulting
rooted distribution is no longer $r$-solvable. So $p(G_2)=19$.
\end{proof}

\begin{lemma}The $r$-critical pebbling number of $G_2$ is 16. \label{critical_goggles}
\end{lemma}

\begin{proof}
The proof is analogous to the proof of Lemma
\ref{glasses_critical}.  The only $r$-critical distributions on
$G_2$ with 16 pebbles have $r=a$ or $r=e$.  We assume $r=e$ and
$D(e)=0$.  Every $r$-critical distribution of this form can be
colored red and blue so that the resulting red and blue rooted
distributions $R$ and $B$ are $r$-critical distributions on
$G_2-e$. Each of these rooted distributions can be solved by first
pebbling all of the pebbles from $a$ to $b$, resulting in rooted
distributions $R'$ and $B'$, which are $r$-critical distributions
on $H$.  By Lemma \ref{cases}, these rooted distributions must
take one of the five forms given in that lemma. Again, $R$ and $B$
can each have at most eight pebbles, and it follows that
$c_r(G_2)=16$.

Note that given the assumption $r=e$, $R$ must either have
$R(a)=8$ or $R(a,f,g,h,i)=(4,1,1,1,1)$, and equivalently for $B$.
It follows that the only $r$-critical distributions on $G_2$ with
16 pebbles must be a combination of two of these rooted
distributions. However, $D(a,f,g,h,i)=(8,2,2,2,2)$ is not
$r$-critical.  Therefore, the only $r$-critical distributions on
$G_2$ with 16 pebbles are either $D(a)=16$ or
$D(a,f,g,h,i)=(12,1,1,1,1)$.
\end{proof}

\begin{theorem}The graph $G_2$ shown in Figure 8 is not thrifty,
not greedy, and has $r$-critical pebbling number $2^d$, where $d$
is the diameter of $G_2$. \label{goggles_theorem}
\end{theorem}

\begin{proof}
By Lemma \ref{critical_goggles}, $c_r(G_2)=16=2^d$, because
$d(G_2)=4$. The first rooted distribution in Figure \ref{goggles}
shows a non-greedy $r$-ceiling distribution on $G_2$, so $G_2$ is
not thrifty. The second rooted distribution in Figure
\ref{goggles} shows a non-greedy rooted distribution on $G_2$ with
19 pebbles, which is the pebbling number of $G_2$ by Lemma
\ref{pebbling_goggles}.
\end{proof}

\begin{figure}[ht!]
\begin{center}
\includegraphics{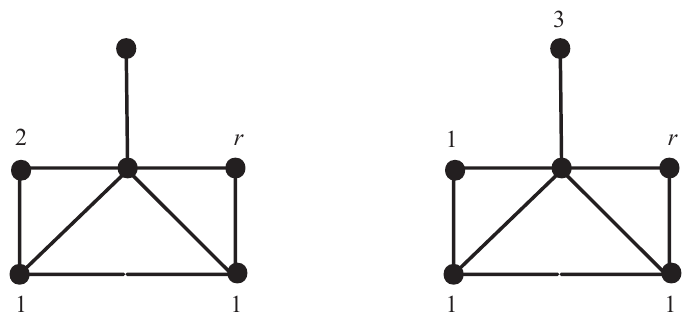}

\caption{An $r$-critical and an $r$-insufficient distribution on
the graph $G_3$.} \label{sting_ray}
\end{center}
\end{figure}

\begin{theorem}The graph $G_3$ shown in Figure \ref{sting_ray} is
greedy, not thrifty, and has $r$-critical pebbling number $2^d$.
\label{sting_ray_theorem}
\end{theorem}

\begin{proof}
The second rooted distribution in Figure \ref{sting_ray} is an
$r$-insufficient distribution on $G_3$ with 6 pebbles.  This
implies that $p(G_3) \geq 7$.

Let $D$ be a rooted distribution on $G_3$ with 7 pebbles.  We show
that $D$ is $r$-solvable and greedy.  If $r$ has a pebble on it,
we are done.  If every vertex but $r$ has a pebble on it, then
there is one vertex $a$ with more than one pebble.  In this case,
we can pebble from $a$ to $r$ using a shortest path, so $D$ is
$r$-solvable and greedy.

Now suppose that there is a vertex $a$ other than $r$ with no
pebbles on it.  That leaves 4 vertices and 7 pebbles.  By the
pigeonhole principle, $D$ has either one vertex with at least 4
pebbles or two vertices $a$ and $b$ with at least 2 pebbles each.  In the first case, the 4 pebbles can be pebbled to $r$ along a shortest path.
In the second case, either $a$ or $b$ is adjacent to $r$ or both $a$ and $b$ are adjacent to a neighbor of $r$.  So again
$D$ is $r$-solvable and greedy.  Since every rooted distribution
on $G_3$ with 7 pebbles is $r$-solvable and greedy, $p(G_3)=7$ and
$G_3$ is greedy.

By Lemma \ref{cases}, any $r$-critical distribution $D$ on $G_3$
with more than 4 pebbles must have at least 2 pebbles on one
vertex and at least one pebble on at least 3 other vertices.  Call
the first vertex $a$ and the other three vertices $b$, $c$, and
$d$. For $D$ to be $r$-critical, the only solution of $D$ must be
$([a,b], [b,c], [c,d], [d,r])$. Then the induced subgraph
consisting of these five vertices must be a path. But there is no
induced $P_5$ in $G_3$. So $c_r(G_3) \leq 4$, and because
$d(G_3)=2$, $c_r(G_3)=4$. In particular, this implies that
$c_r(G_3)=2^d$.

The first rooted distribution $D$ shown in Figure \ref{sting_ray}
is $r$-critical and has four pebbles, and so $D$ is an $r$-ceiling
distribution. Therefore, since $D$ is not greedy, $G_3$ is not
thrifty.
\end{proof}

\begin{figure}[ht!]
\begin{center}
\includegraphics{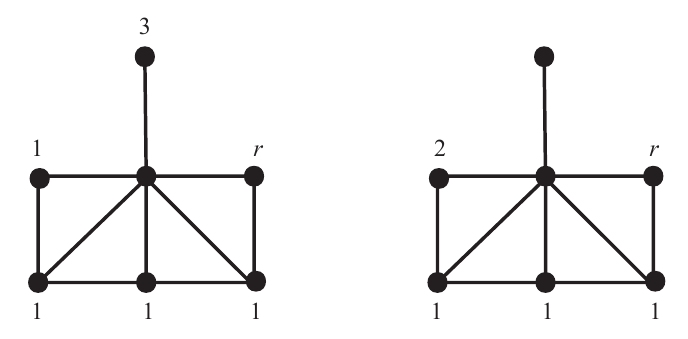}

\caption{An $r$-insufficient and an $r$-critical distribution on
the graph $G_4$.} \label{romulan}
\end{center}
\end{figure}

\begin{theorem}
The graph $G_4$ shown in Figure \ref{romulan} is greedy, has $c_r(G) > 2^d$, and is not thrifty.
\label{romulan_theorem}
\end{theorem}

\begin{proof} It suffices to show that $G_4$ is
greedy and has $r$-critical pebbling number greater than $2^d$; then by Theorem~\ref{thrifty_number}, $G_4$ is not thrifty.

The proof is analogous to the proof of Theorem
\ref{sting_ray_theorem}.  The first distribution in Figure
\ref{romulan} is an $r$-insufficient distribution on $G_4$ with 7
pebbles.  This implies that $p(G_4) \geq 8$.

Let $D$ be a rooted distribution on $G_4$ with 8 pebbles.  As
before, we can rule out the case in which $D$ has a pebble on $r$
and the case in which there is at least one pebble on every other
vertex. Consider the remaining case, in which there is a vertex
$a$ other than $r$ with no pebbles on it. That leaves 5 vertices
and 8 pebbles. By the pigeonhole principle, $D$ has either one
vertex with at least 4 pebbles or two vertices with at least 2
pebbles each. As before, every rooted distribution on $G_4$ with 8
pebbles is $r$-solvable and greedy, so $p(G_4)=8$ and $G_4$ is
greedy.

Again by Lemma \ref{cases}, any $r$-critical distribution $D$ on
$G_4$ with more than 5 pebbles must have 2 pebbles on one vertex
and one pebble on at least 4 other vertices.  By the same
reasoning as in Theorem \ref{sting_ray_theorem}, the induced
subgraph consisting of these six vertices must be a path. Since
there is no induced $P_6$ in $G_4$, $c_r(G_4) \leq 5$.  The second
rooted distribution $D$ shown in Figure \ref{romulan} is
$r$-critical and has five pebbles, so $c_r(G_4)=5$.  In
particular, $c_r(G_4)>4=2^d$, since $d(G_4)=2$.
\end{proof}

\begin{figure}[ht!]
\begin{center}
\includegraphics{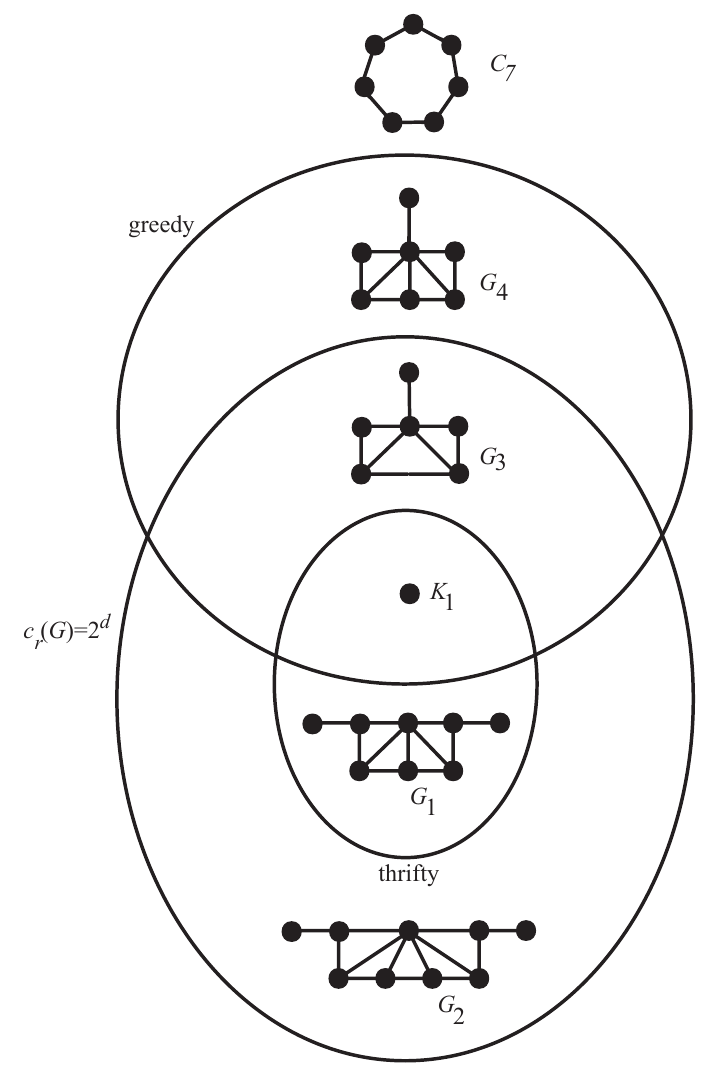}

\caption{The Venn diagram of rooted pebbling properties of
graphs.} \label{pebbling_classes}
\end{center}
\end{figure}

The above results are summarized in the Venn diagram in Figure
\ref{pebbling_classes}.  In each region of this Venn diagram, a
graph is shown with the given properties.

\begin{theorem}The fan $F_k$ has weight $\frac{k+1}{4}$. \end{theorem}

\begin{proof}
By Theorem \ref{wheel_theorem}, $c_r(F_k)=k$.  Thus, the weight of
$F_k$ is the largest weight of an $r$-critical distribution $D$ on
$F_k$ with $k$ pebbles.  No $r$-critical distribution with more
than two pebbles can have more than one pebble in the neighborhood
of $r$. Since $d(F_k)=2$, this means that at least $k-1$ of the
$k$ pebbles in $D$ are distance 2 from $r$.  If all $k$ pebbles in
$D$ are distance 2 from $r$, then $w(D)=\frac{k}{4}$.  If exactly
$k-1$ pebbles in $D$ are distance 2 from $r$ then
$w(D)=\frac{1}{2}+\frac{k-1}{4}=\frac{k+1}{4}$.
%
The first diagram in Figure \ref{broken_wheel} shows an
$r$-critical distribution with $k$ pebbles and weight
$\frac{k+1}{4}$.
\end{proof}

\begin{corollary}There exist graphs with diameter 2 that have arbitrarily
large weight. \hfill $\square$ \end{corollary}

This concludes our discussion of the $r$-critical pebbling number.
 We hope to explore the $g$-critical and $u$-critical pebbling
numbers further in future work.

\section*{Acknowledgements}

We thank Glenn Hurlbert for suggesting the $g$-critical pebbling
number to us, and Glenn Hurlbert, Aparna Higgins, and an anonymous referee for many helpful suggestions.

\bibliographystyle{amsplain}

\end{document}